\documentclass[12pt]{article}

\usepackage{amsmath}
\usepackage{amssymb}
\usepackage{amsthm}
\usepackage{latexsym}
\usepackage{cite}

\newtheorem{theorem}{Theorem}[section]
\newtheorem{lemma}[theorem]{Lemma}
\newtheorem{corollary}[theorem]{Corollary}
\newtheorem{proposition}[theorem]{Proposition}
\theoremstyle{definition}
\newtheorem{definition}[theorem]{Definition}

\makeatletter \@addtoreset{equation}{section} \makeatother

\begin{document}

\title{On upper bounds on the smallest size of a saturating set in a
projective plane\thanks{The research of  D. Bartoli, M. Giulietti, S. Marcugini, and F.~Pambianco was
 supported in part by Ministry for Education, University
and Research of Italy (MIUR) (Project ``Geometrie di Galois e
strutture di incidenza'')
 and by the Italian National Group for Algebraic and Geometric Structures and their Applications
 (GNSAGA - INDAM).
 The research of A.A.~Davydov was carried out at the IITP RAS at the expense of the Russian
Foundation for Sciences (project 14-50-00150).}}
\date{}
\maketitle
\begin{center}
{\sc Daniele Bartoli}\\
{\sc\small Dipartimento di Matematica  e Informatica,
   Universit{\`a} degli Studi di Perugia}\\
{\sc\small Perugia, 06123, Italy}\\ \emph {E-mail address:} daniele.bartoli@unipg.it\medskip\\
{\sc Alexander A. Davydov}\\
{\sc\small Institute for Information Transmission Problems (Kharkevich institute)}\\
 {\sc\small Russian Academy of
 Sciences}\\ {\sc\small GSP-4, Moscow, 127994, Russian Federation}\\\emph {E-mail address:} adav@iitp.ru\medskip\\
 {\sc Massimo Giulietti, Stefano Marcugini, Fernanda Pambianco}\\
 {\sc\small Dipartimento di Matematica  e Informatica,
  Universit{\`a} degli Studi di Perugia}\\
 {\sc\small Perugia, 06123, Italy}\\
 \emph{E-mail address:} massimo.giulietti, stefano.marcugini, fernanda.pambianco@unipg.it
\end{center}
\begin{abstract}
In a projective plane $\Pi _{q}$ (not necessarily Desarguesian)
of order $q,$ a point subset $S$ is saturating (or dense) if
any point of $\Pi _{q}\setminus S$ is collinear with two points
in$~S$. Using probabilistic methods, the following upper bound
on the smallest size $ s(2,q)$ of a saturating set in $\Pi
_{q}$ is proved:
\begin{equation*}
s(2,q)\leq 2\sqrt{(q+1)\ln (q+1)}+2\thicksim 2\sqrt{q\ln q}.
\end{equation*}
We also show that for any constant $c\ge 1$ a random point set of size $k$
in $\Pi _{q}$ with $
2c\sqrt{(q+1)\ln(q+1)}+2\le k<\frac{q^{2}-1}{q+2}\thicksim q$ is a saturating
set with probability greater than $1-1/(q+1)^{2c^{2}-2}.$ Our
probabilistic approach is also applied to multiple saturating
sets. A point set $S\subset \Pi_{q}$ is $(1,\mu)$-saturating if
for every point $Q$ of $\Pi _{q}\setminus S$ the number of
secants of $S$ through $Q$ is at least $\mu $, counted with
multiplicity. The multiplicity of a secant $ \ell $ is computed
as ${\binom{{\#(\ell \,\cap S)}}{{2}}}.$ The following upper
bound on the smallest size $s_{\mu }(2,q)$ of a $(1,\mu)$-saturating set in $\Pi_{q}$ is proved:
\begin{equation*}
s_{\mu }(2,q)\leq 2(\mu +1)\sqrt{(q+1)\ln (q+1)}+2\thicksim 2(\mu +1)\sqrt{
q\ln q}\,\text{ for }\,2\leq \mu \leq \sqrt{q}.
\end{equation*}

By using inductive constructions,
 upper bounds on the smallest size
of a saturating set (as well as on a $(1,\mu)$-saturating set) in
the projective space $PG(N,q)$ are obtained.

All the results are also stated in terms of linear covering codes.
\end{abstract}

\section{Introduction}

We denote by $\Pi _{q}$ a projective plane (not necessarily
Desarguesian) of order $q$ and by $PG(2,q)$ the projective
plane over the Galois field with $q$ elements.

\begin{definition}
\label{def1_usual satur} A point set $S\subset \Pi _{q}$ is
\emph{saturating} if any point of $\Pi _{q}\setminus S$ is
collinear with two points in $S$.
\end{definition}

Saturating sets are considered, for example, in \cite
{BFMP-JG2013,BorSzTic,BrPlWi,DavCovRad2,DGMP-AMC,DMP-JCTA2003,DavOst-EJC,FainaGiul,GacsSzonyi,%
Giul-plane,Giul2013Survey,GiulTor,Janwa1990,Kovacs,MP_Austr2003,Ughi};
see also the references therein. It should be noted that
saturating sets are also called ``saturated
sets'' \cite
{DavCovRad2,Janwa1990,Kovacs,Ughi}, ``spanning
sets'' \cite{BrPlWi}, ``dense sets''
\cite{BorSzTic,FainaGiul,GacsSzonyi,Giul-plane,GiulTor}, and ``1-saturating sets''
\cite{BDGMP-AMC2015,BDGMP-arXiv2015,DGMP-AMC,DMP-JCTA2003,DavOst-EJC}.

The homogeneous coordinates of the points of a saturating set
of size $k$ in $PG(2,q)$ form a parity check matrix of a
$q$-ary linear code with length $k,$ codimension 3, and
covering radius 2. For an introduction to covering codes see
\cite{Handbook-coverings,CHLS-bookCovCod}. An online
bibliography on covering codes is given in \cite{LobstBibl}.

The main problem in this context is to find small saturating
sets (i.e. short covering codes). Let $s(2,q)$ denote the
smallest size of a saturating set in $\Pi _{q}$. In
\cite{BorSzTic}, by using the probabilistic approach
previously introduced in \cite{Kovacs}, the following upper
bound is proved:
\begin{equation}
s(2,q)<3\sqrt{2}\sqrt{q\ln q}<5\sqrt{q\ln q}.  \label{eq1_1sat}
\end{equation}

Surveys on random constructions for geometrical objects can be
found in \cite {BorSzTic,GacsSzonyi,KV,Kovacs}; see also the
references therein. Saturating sets in $PG(2,q)$ obtained by
algebraic constructions or computer search can be found in
\cite
{BFMP-JG2013,BrPlWi,DavCovRad2,DGMP-AMC,DMP-JCTA2003,DavOst-EJC,FainaGiul,Giul-plane,%
Giul2013Survey,GiulTor,MP_Austr2003,Ughi}.

In this paper, we use probabilistic methods to obtain new upper
bounds on $s(2,q)$. Our main results are Theorems \ref{th1} and
\ref{th2} below.

\begin{theorem}
\label{th1} For the smallest size $s(2,q)$ of a saturating set
in a projective plane of order $q$ the following upper bound
holds:
\begin{equation}
s(2,q)\leq 2\sqrt{(q+1)\ln (q+1)}+2\thicksim 2\sqrt{q\ln q}.
\label{eq1_satsetsize}
\end{equation}
\end{theorem}

\begin{theorem}
\label{th2} Let $c$ be a real number greater than or equal to $1$ and
 let $k$ be an integer such that
\begin{equation*}
2c\sqrt{(q+1)\ln (q+1)}+2\le k<\frac{q^{2}-1}{q+2}\sim q.
\end{equation*}
Then in a projective plane  of order $q$ a random point set of
size $k$ is a saturating set with probability grater than
\begin{equation}
1-\frac{1}{(q+1)^{2c^{2}-2}}\,.  \label{eq1_bigprob}
\end{equation}
\end{theorem}

Theorem \ref{th1} improves the constant term of
\eqref{eq1_1sat}. It should be noted that our approach is
different from  those in \cite{BorSzTic,Kovacs}, where random
sets lying on two or three lines are considered; in this paper
arbitrary random sets are dealt with.

Theorem \ref{th1} can be expressed in terms of \emph{covering codes}.
The \emph{length function} $\ell (R,r,q)$ denotes the smallest
length of a $ q $-ary linear code with covering radius $R$ and
codimension~$r$; see \cite
{Handbook-coverings,BrPlWi,CHLS-bookCovCod}. Theorem~\ref{th1}
can be read as follows.

\begin{corollary}
The following upper bound on the length function holds.
\begin{equation*}
\ell (2,3,q)\leq 2\sqrt{(q+1)\ln (q+1)}+2\thicksim 2\sqrt{q\ln q}.
\end{equation*}
\end{corollary}

Our probabilistic approach can also be applied to
\emph{multiple saturating} sets.

\begin{definition}
\label{def2_mu_satur}A point set $S\subset \Pi _{q}$ is $(1,\mu)
$-\emph{saturating} if for every point $Q$ of $\Pi
_{q}\setminus S$ the number of secants of $S$ through $Q$ is at
least $\mu $, counted with multiplicity. Here the multiplicity
of a secant $\ell $ is computed as ${\binom{{\#(\ell \,\cap
S)}}{{2}}}.$
\end{definition}

For $\mu =1,$ a $(1,\mu)$-saturating set is a saturating set as in
Definition \ref{def1_usual satur}.

The homogeneous coordinates of the points of a $(1,\mu)
$-saturating set of size $k$ in $PG(2,q)$
 form a
parity check matrix of a $q$-ary linear code with length $k,$
codimension 3, covering radius 2. Such a code is a $(2,\mu
)$-multiple covering of the farthest-off-points ($ (2,\mu
)$-MCF code or simply MCF code, for short). For an introduction
to multiple saturating sets and MCF codes see \cite[Chapters
13,\thinspace 14] {CHLS-bookCovCod}, \cite
{BDGMP-AMC2015,BDGMP-arXiv2015,Giul2013Survey,HonkLits1996,PDBGM-ExtendAbstr},
and references therein.

The main problem in this context is to find small $(1,\mu)
$-saturating sets (i.e. short MCF codes). Let $s_{\mu }(2,q)$
be the smallest size of a $(1,\mu) $-saturating set in $\Pi _{q}$.
Our main results on $(1,\mu)$-saturating sets in $\Pi _{q}$ are the
following.

\begin{theorem}
\label{th3} Let $\mu \geq 2.$ For the smallest size $s_{\mu
}(2,q)$ of a $(1,\mu)$-saturating set in a projective plane of order $q$ the following upper
bounds hold.
\begin{equation}
s_{\mu }(2,q)\leq 2D_{\mu }\sqrt{(q+1)\ln (q+1)}+2\thicksim 2D_{\mu }\sqrt{
q\ln q},  \label{eq1_mu_satsetsize}
\end{equation}
where
\begin{equation}
D_{\mu }\leq \left\{
\begin{array}{cccc}
2.4 & \text{for} & \mu =2, & q\geq 97 \\
2.6 & \text{for} & \mu =3, & q\geq 181 \\
2.8 & \text{for} & \mu =4, & q\geq 125 \\
\mu +1 & \text{for} & \mu \leq \sqrt{q}, & q\geq 4 \\
2\mu -1 & \text{for} & \mu \leq \frac{1}{2}((1-\delta )q-\delta +1)+1, &
q\geq 3
\end{array}
\right. ,  \label{eq1_Dmu}
\end{equation}
\begin{equation}
\delta =\frac{1}{\sqrt{(q+1)\ln \left( q+1\right) }}.  \label{eq1_delta}
\end{equation}
\end{theorem}

The $\mu $-\emph{length function} $\ell _{\mu }(R,r,q)$ denotes
the smallest length of a linear $q$-ary $(R,\mu )$-MCF code
with covering radius $R$ and codimension~$r$
\cite{BDGMP-AMC2015,BDGMP-arXiv2015,Giul2013Survey,PDBGM-ExtendAbstr}.
For $ \mu =1,$ $\ell _{1}(R,r,q)$ is the usual length function
$\ell (R,r,q)$ for 1-fold coverings. In the \emph{covering code
language,} Theorem~\ref{th3} can be read as follows.

\begin{corollary}
Let $D_{\mu }$ be as in \eqref{eq1_Dmu}. The following upper
bound on the $ \mu $-length function holds.
\begin{equation}
\ell _{\mu }(2,3,q)\leq 2D_{\mu }\sqrt{(q+1)\ln (q+1)}+2\thicksim 2D_{\mu }
\sqrt{q\ln q}.  \label{eq1_ell_mu_Dmu}
\end{equation}
\end{corollary}

In \cite[Prop.\thinspace 5.2]{BDGMP-AMC2015} the following
upper bounds on the $\mu$-length function were obtained by
adapting the probabilistic approach in \cite{BorSzTic,Kovacs}:
\begin{equation}
s_{\mu }(2,q)\leq \ell _{\mu }(2,3,q)<66\sqrt{\mu q\ln q}\text{ for }\mu
<121q\ln q.  \label{eq1_bound66}
\end{equation}

The bounds (\ref{eq1_mu_satsetsize}) and (\ref{eq1_ell_mu_Dmu})
improve (\ref{eq1_bound66}) provided that $D_{\mu }<33\sqrt{\mu
}.$ This actually happens for a wide region of $\mu ,$ see
(\ref{eq1_Dmu}).

Let $PG(N,q)$ be the $N$-dimensional projective space over the
Galois field of $q$ elements.

From (\ref{eq1_satsetsize}) and (\ref {eq1_mu_satsetsize}), by
using inductive constructions from \cite
{BDGMP-AMC2015,DavCovRad2,DGMP-AMC}, upper bounds on the
smallest size of a saturating set in the $N$-dimensional
projective space $PG(N,q)$ can be obtained; see Section
\ref{sec_space}. In many cases these bounds are better than the
known ones.

The paper is organized as follows. In Section
\ref{sec_1-satur}, we deal with upper bounds on the smallest
size of a saturating set in a projective plane; Theorems
\ref{th1} and \ref{th2} are proved using probabilistic methods. In
Sections \ref{sec_mu-satur} and \ref{sec_mu-satur improv} we
apply our probabilistic approach to $(1,\mu)$-saturating sets in a
projective plane. Finally, in Section \ref{sec_space}, bounds
for saturating and $(1,\mu)$-saturating sets in the projective
space $PG(N,q)$ are obtained.

\section{Upper bound on the smallest size of a saturating set in a
projective plane\label{sec_1-satur}}

Let $w>0$ be a fixed integer. Consider a random $(w+1)$-point
subset $\mathcal{K} _{w+1}$ of $\Pi _{q}.$ The total number of
such subsets is $\binom{q^{2}+q+1 }{w+1}.$ A fixed point $A$ of
$\Pi _{q}$ is \emph{covered} by $\mathcal{K} _{w+1}$ if it
belongs to an $r$-secant of $\mathcal{K}_{w+1}$ with $r\geq 2.$
We denote by $\mathrm{Prob}(\diamond )$ the probability of some
event~$ \diamond $.

We estimate
\begin{equation*}
\pi :={\mathrm{Prob}}(A\text{ not covered by }\mathcal{K}_{w+1})
\end{equation*}
as the ratio of the number of $(w+1)$-point subsets not
covering $A$ over the total number of subsets of size $(w+1)$.
Since a set $\mathcal{K}_{w+1}$ does not cover $A$ if and only
if every line through $A$ contains at most one point of
$\mathcal{K}_{w+1}$, we have
\begin{equation}
\pi =\frac{q^{w+1}\binom{q+1}{w+1}}{\binom{q^{2}+q+1}{w+1}}.
\label{eq2_pi=_start}
\end{equation}

By straightforward calculations,
\begin{equation}
\pi =\frac{(q^{2}+q)(q^{2})\cdots (q^{2}+q-iq)\cdots (q^{2}+q-wq)}{
(q^{2}+q+1)(q^{2}+q)\cdots (q^{2}+q-i)\cdots (q^{2}+q+1-w)}=
\label{eq2_calcul}
\end{equation}
\begin{equation*}
=\prod_{i=0}^{w}\frac{q^{2}+q-iq}{q^{2}+q+1-i}=\prod_{i=0}^{w}\left( 1-\frac{
iq-i+1}{q^{2}+q+1-i}\right) <\prod_{i=0}^{w}\left( 1-\frac{i(q-1)}{q^{2}+q+1}
\right) .
\end{equation*}

Using the inequality $1-x\leq e^{-x}$ we obtain that
\begin{equation*}
\pi <e^{-\sum_{i=0}^{w}\frac{i(q-1)}{q^{2}+q+1}}=e^{-\frac{(w^{2}+w)(q-1)}{
2(q^{2}+q+1)}},
\end{equation*}
which implies
\begin{equation}
\pi <e^{-\frac{(w^{2}+w)(q-1)}{2(q^{2}+q+1)}}<e^{-\frac{w^{2}}{2q+2}},
\label{eq2_pi_estimate}
\end{equation}
provided that
\begin{equation*}
\frac{(w+1)(q^{2}-1)}{(q^{2}+q+1)}>w
\end{equation*}
that is
\begin{equation*}
w<\frac{q^{2}-1}{q+2}\sim q.
\end{equation*}

The set $\mathcal{K}_{w+1}$ is not saturating if at least one
point $A\in \Pi _{q}$ is not covered by $\mathcal{K}_{w+1}.$
Similarly to \cite[Proposition 4.1]{BorSzTic}, we note that
\begin{equation}
\mathrm{Prob}\left( \mathcal{K}_{w+1}\text{ is not saturating}\right)
\leq \sum\limits_{A\in \Pi _{q}}\mathrm{Prob}(A\text{ is not covered}).
\label{eq2_Kw+1_not_satur}
\end{equation}
Now, using \eqref{eq2_pi_estimate}, we obtain that
\begin{equation}
\mathrm{Prob}\left( \mathcal{K}_{w+1}\text{ is not saturating}\right)
\leq (q^{2}+q+1)\pi <(q+1)^{2}e^{-\frac{w^{2}}{2q+2}}.
\label{eq2_Kw+1_not_satur_2}
\end{equation}

Therefore, the probability that all the points of $\Pi _{q}$ are covered is
\begin{equation}
\mathrm{Prob}\left( \mathcal{K}_{w+1}\text{ is saturating}\right)
>1-(q+1)^{2}e^{-\frac{w^{2}}{2q+2}}.  \label{eq2_A_covered}
\end{equation}
This quantity is larger than $0$ taking, for instance,
\begin{equation*}
w=\left\lceil \sqrt{(2q+2)\ln \left( (q+1)^{2}\right) }\right\rceil .
\end{equation*}
This shows that in $\Pi _{q}$ there exists a saturating set
with size
\begin{equation*}
k\leq 2\sqrt{(q+1)\ln (q+1)}+2\thicksim 2\sqrt{q\ln q}.
\end{equation*}
and therefore Theorem \ref{th1} is proved.

In conclusion, we note that any value
\begin{equation*}
w=\left\lceil c\sqrt{(2q+2)\ln \left( (q+1)^{2}\right) }\right\rceil <\frac{
q^{2}-1}{q+2}
\end{equation*}
where the parameter $c\geq 1$ is independent of $q$, provides in
\eqref{eq2_A_covered} a positive probability greater than $1-1/(q+1)^{2c^{2}-2}$; therefore Theorem~\ref{th2}
holds.

It is worth noting that in \eqref{eq1_bigprob} for $q$ large
enough, choosing $c=1+\varepsilon $, with $\varepsilon
=o(1)>0$, the probability is close to 1.

\section{Upper bounds on the smallest size of a $(1,\mu) $-satura\-ting
set in a projective plane, $\mu \geq 2$\label{sec_mu-satur}}

For $\mu \geq 2,$ we construct a $(1,\mu) $-saturating set
$\mathcal{S}_{\mu }$ in $\Pi _{q}$ by joining a $(1,\mu
-1)$-saturating set $\mathcal{S} _{\mu -1}$ and a
\textquotedblleft usual\textquotedblright $\ $saturating set
\emph{disjoint} from $\mathcal{S}_{\mu -1}$.

Let $w>0$ be a fixed integer; we consider a random $(w+1)$-point
subset $\mathcal{H} _{w+1}$ of $\Pi _{q}$ disjoint from
$\mathcal{S}_{\mu -1}$. Let $k$ denote the size of $\mathcal
S_{\mu-1}$. Then the total number of such subsets is
$\binom{q^{2}+q+1-k}{w+1}.$

Clearly, if $\mathcal{H}_{w+1}$ is a saturating set then
$\mathcal{S}_{\mu -1}\cup \mathcal{H}_{w+1}$ is a $(1,\mu)
$-saturating set.

We argue as in Section \ref{sec_1-satur}. For a fixed point $A$
of $\Pi _{q}$ we estimate
\begin{equation*}
\lambda :={\mathrm{Prob}}(A\text{ not covered by }\mathcal{H}_{w+1})
\end{equation*}
as the ratio of the number of $(w+1)$-point subsets not
covering $A$ and disjoint from $\mathcal{S}_{\mu -1}$ over the
total number of subsets of size $(w+1)$ disjoint from
$\mathcal{S}_{\mu -1}$. Similarly to (\ref{eq2_pi=_start}), we
have
\begin{equation}
\lambda <\frac{q^{w+1}\binom{q+1}{w+1}}{\binom{q^{2}+q+1-k}{w+1}}.
\label{eq3_lambda_main}
\end{equation}
In fact, the number of $(w+1)$-point subsets not covering $A$
and disjoint from $ \mathcal{S}_{\mu -1}$ is smaller than the
numerator of (\ref{eq3_lambda_main}).

By straightforward calculations similar to (\ref{eq2_calcul}),
\begin{align*}
&\lambda <\frac{(q^{2}+q)(q^{2})\cdots (q^{2}+q-wq)}{(q^{2}+q+1-k)(q^{2}+q-k)
\cdots (q^{2}+q+1-w-k)}<
\prod_{i=0}^{w}\left( 1-\frac{i(q-1)-k}{q^{2}+q+1}
\right) .
\end{align*}
Now, under the condition $w<\frac{q^{2}-1}{q+2}\sim q,$ as in
(\ref{eq2_pi_estimate}) we obtain that
\begin{equation*}
\lambda <e^{-\sum_{i=0}^{w}\frac{i(q-1)-k}{q^{2}+q+1}}<e^{-\frac{
(w^{2}+w)(q-1)}{2(q^{2}+q+1)}+\frac{k(w+1)}{2q(q+1)}}<e^{-\frac{w^{2}}{2q+2}+
\frac{k(w+1)}{2q(q+1)}}.
\end{equation*}
This implies that
\begin{equation*}
\mathrm{Prob}\left( \mathcal{H}_{w+1}\text{ is not saturating}\right)
\leq (q^{2}+q+1)\lambda <(q+1)^{2}e^{-\frac{w^{2}}{2q+2}+\frac{k(w+1)}{
2q(q+1)}}.
\end{equation*}
So,
\begin{equation}
\mathrm{Prob}\left( \mathcal{S}_{\mu -1}\cup \mathcal{H}_{w+1}\text{ is }(1,\mu)
\text{-saturating}\right) >1-(q+1)^{2}e^{-\frac{w^{2}}{2q+2}+\frac{k(w+1)
}{2q(q+1)}}.  \label{eq3_Tw+1 satur}
\end{equation}

Throughout this section, $\delta $ is as in (\ref{eq1_delta}).
We represent $ w$ and $k$ in the following form:
\begin{equation}
w=\left\lceil d\sqrt{(2q+2)\ln (\left( q+1\right) ^{2})}\right\rceil ,\text{
}d>1\text{ independent of }q;  \label{w =}
\end{equation}
\begin{equation*}
k\leq 2D\sqrt{(q+1)\ln \left( q+1\right) }+2,\text{ }D\geq 1\text{
independent of }q.
\end{equation*}
Then the following holds:
\begin{equation}
w+1\leq 2(d+\delta )\sqrt{(q+1)\ln \left( q+1\right) };  \label{w+1<=}
\end{equation}
\begin{equation}
k\leq 2(D+\delta )\sqrt{(q+1)\ln \left( q+1\right) };  \label{k<=}
\end{equation}
\begin{equation*}
(q+1)^{2}e^{-\frac{w^{2}}{2q+2}+\frac{k(w+1)}{2q(q+1)}}<\frac{
(q+1)^{2(D+\delta )(d+\delta )/q}}{(q+1)^{2d^{2}-2}}=(q+1)^{\frac{2}{q}
(D+\delta )(d+\delta )-2(d^{2}-1)}.
\end{equation*}
Let
\begin{equation}
d=1+\frac{D+\delta }{q}.  \label{d>=}
\end{equation}
Then
\begin{align*}
&d-1=\frac{D+\delta }{q}>\frac{2(D+\delta )(d+\delta )}{2q(d+1)};\\
&\frac{2}{q}(D+\delta )(d+\delta )-2(d^{2}-1) <0; \\
&(q+1)^{2}e^{-\frac{w^{2}}{2q+2}+\frac{k(w+1)}{2q(q+1)}} <1.
\end{align*}
The last inequality means that the probability in (\ref{eq3_Tw+1 satur}) is
positive. As $\#(\mathcal{S}_{\mu -1}\cup \mathcal{H}_{w+1})=k+w+1,$ taking
into account (\ref{eq1_delta}), (\ref{w+1<=})--(\ref{d>=}), we have proved
the following lemma.

\begin{lemma}
\label{lem_extention_mu-1} Let $\Pi _{q}$ be a projective plane
 of order $q$. Let $\mu \geq 2$
and assume that for some $D\ge 1$ in $\Pi _{q}$ there exists a
$(1,\mu -1)$-saturating set with size $k\leq 2D\sqrt{(q+1)\ln
(q+1)}+2$. Then in $\Pi _{q}$ there exists a $(1,\mu)$-saturating
set with size
\begin{equation}
v\leq 2\left( D+1+\frac{D+\delta }{q}+\delta \right) \sqrt{(q+1)\ln \left(
q+1\right) }+2.  \label{main}
\end{equation}
\end{lemma}

\begin{corollary}
\label{cor3}Let $\mu \geq 2.$

\begin{description}
\item[(i)] In $\Pi _{q}$ there is a $(1,\mu) $-saturating set
    with size
\begin{equation*}
k\leq 2D_{\mu }\sqrt{(q+1)\ln (q+1)}+2,
\end{equation*}
where
\begin{equation}
D_{1}=1,\quad D_{i}=D_{i-1}+1+\frac{D_{i-1}+\delta }{q}+\delta ,\text{ }
i=2,3,\ldots \mu .  \label{Di}
\end{equation}

\item[(ii)] In $\Pi _{q}$ there is a $(1,\mu) $-saturating set
    with size
\begin{equation*}
k\leq 2(\mu +1)\sqrt{(q+1)\ln (q+1)}+2,\quad \mu \leq \sqrt{q}.
\end{equation*}

\item[(iii)] In $\Pi _{q}$ there is a $(1,\mu) $-saturating set
    with size
\begin{equation*}
k\leq 2(2\mu -1)\sqrt{(q+1)\ln (q+1)}+2
\end{equation*}
provided that
\begin{equation*}
\mu \leq \frac{(1-\delta )q-\delta +1}{2}+1.
\end{equation*}
\end{description}
\end{corollary}

\begin{proof}

\begin{description}
\item[(i)] For $\mu =1,$ a $(1,\mu)$-saturating set is an usual saturating set.
Therefore, we may use Theorem \ref{th1} and put $D_{1}=1$. Then we
iteratively apply (\ref{main}).

\item[(ii)] We use (i). By (\ref{Di}),
\begin{equation}
D_{i}=i+\frac{\sum_{j=1}^{i-1}D_{j}+(i-1)\delta }{q}+(i-1)\delta ,\text{ }
i=2,3,\ldots \mu .  \label{Di=i+}
\end{equation}
By induction, we will show that, under the condition $\mu
\leq \sqrt{q}$,
\begin{equation*}
A_{i}:=\frac{\sum_{j=1}^{i-1}D_{j}+(i-1)\delta }{q}+(i-1)\delta \leq 1
\end{equation*}
holds for $i=2,3,\ldots, \mu$. We have
\begin{equation*}
A_{2}=\frac{D_{1}+\delta }{q}+\delta =\frac{1+\delta }{q}+\delta <1.
\end{equation*}
Assume that $A_{i}\leq 1,$ $i=2,3,\ldots ,h,$ with $h\leq
\mu -1\leq \sqrt{q} -1.$ Then, by (\ref{Di}),
(\ref{Di=i+}), we have that
\begin{equation*}
D_{1}=1;\quad D_{i}\leq i+1,\text{ }i=2,3,\ldots ,h;
\end{equation*}
\begin{equation*}
A_{h+1}\leq \frac{1+\sum_{j=2}^{h}(j+1)+h\delta }{q}+h\delta =\frac{
h^{2}+3h-2+2h\delta }{2q}+h\delta .
\end{equation*}
It can be checked that $A_{h+1}<1$ if $h\leq \sqrt{q}-1.$ So, $D_{i}\leq
i+1, $ $i=2,3,\ldots \mu .$

\item[(iii)] By the proof of (ii), $D_{2}\leq 3$ holds.
    Assume that $D_{i}\leq 2i-1,$ $ i=2,3,\ldots ,h,$ with
    $h\leq \mu
-1\leq \frac{1}{2}((1-\delta )q-\delta +1)$ . Then
\begin{align*}
&D_{h+1} =D_{h}+1+\frac{D_{h}+\delta }{q}+\delta \leq 2h-1+1+\frac{
2h-1+\delta }{q}+\delta \leq \\
&2h+\frac{(1-\delta )q-\delta +1-1+\delta }{q}+\delta =
2h+(1-\delta)+\delta =2h+1.
\end{align*}
\end{description}
\end{proof}

\section{Improved upper bounds on the smallest size of a $(1,\mu)$-saturating set
 in a projective plane,
 $\mu =2,3,4$\label{sec_mu-satur improv}}

Let $w>0$ be a fixed integer. We consider a random $(w+1)$-point
subset $\mathcal{K} _{w+1}$ of $\Pi _{q}.$ The total number of
such subsets is $\binom{q^{2}+q+1 }{w+1}.$ As above, let $A$ be
a fixed point of $\Pi _{q}.$

We say that $\mathcal{K}_{w+1}$ covers $A$ \emph{exactly} $i$
times if the number of secants of $\mathcal{K}_{w+1}$ through
$A$ is \emph{exactly} $i$, counted with multiplicity. Denote by
$T_{i}$ the number of $(w+1)$-subsets covering $A$ exactly $i$
times, $i=0,1,2,\ldots ,$ where $i=0$ means that $A$ is not
covered by $\mathcal{K}_{w+1}$. Similarly to the numerator of
(\ref{eq2_pi=_start}) we have
\begin{equation}
T_{0}=q^{w+1}\binom{q+1}{w+1}.  \label{T0}
\end{equation}
According to
 Definition \ref{def2_mu_satur}, we say that a fixed point $A$
of $\Pi _{q}$ is $\mu $-\emph{covered} by $\mathcal{K}_{w+1}$ if the number
of secants of $\mathcal{K}_{w+1}$ through $A$ is at least $\mu $, counted
with multiplicity. We estimate
\begin{equation*}
\pi _{\mu }:={\mathrm{Prob}}(A\text{ not }\mu \text{-covered by }\mathcal{K}
_{w+1})
\end{equation*}
as the ratio of the number of $(w+1)$-point subsets that do not
$\mu $-cover $A$ over the total number of subsets of size
$(w+1)$. So,
\begin{equation}
\pi _{\mu }=\frac{\sum_{i=0}^{\mu -1}T_{i}}{\binom{q^{2}+q+1}{w+1}}
=R_{w,q}\sum\limits_{i=0}^{\mu -1}\frac{T_{i}}{T_{0}}  \label{pi_1,mu}
\end{equation}
where
\begin{equation}
R_{w,q}=\frac{T_{0}}{\binom{q^{2}+q+1}{w+1}}.  \label{Rwq}
\end{equation}

The set $\mathcal{K}_{w+1}$ is not $(1,\mu) $-saturating if at
least one point $ A\in \Pi _{q}$ is not $\mu $-covered by
$\mathcal{K}_{w+1}.$ As in (\ref{eq2_Kw+1_not_satur}) and
(\ref{eq2_Kw+1_not_satur_2}), we have
\begin{equation*}
\mathrm{Prob}\left( \mathcal{K}_{w+1}\text{ is not }(1,\mu) \text{-saturating}\right) \leq (q^{2}+q+1)\pi _{\mu }<(q+1)^{2}\pi _{\mu }.
\end{equation*}
Hence, the probability that all the points of $\Pi _{q}$ are $\mu $-covered
is
\begin{equation}
\mathrm{Prob}\left( \mathcal{K}_{w+1}\text{ is a }(1,\mu) \text{-saturating}
\right) \geq 1-(q^{2}+q+1)\pi _{\mu }>1-(q+1)^{2}\pi _{\mu }.
\label{K_w+1  satur}
\end{equation}

Throughout this section, we represent $w$ in the form (\ref{w
=}). Also, from now, we assume
\begin{equation}
w<\frac{q+1}{2}.  \label{w<(q+1)/2}
\end{equation}

\begin{theorem}
\label{th4} For the smallest size $s_{\mu }(2,q)$ of a $(1,\mu)
$-saturating set in a projective plane of order $q$ the
following upper bounds hold:
\begin{eqnarray}
s_{2}(2,q) &\leq &2.4\sqrt{(q+1)\ln (q+1)}+2\thicksim 2.4\sqrt{q\ln q},\quad
q\geq 97;  \label{eq4_mu=2_bound} \\
s_{3}(2,q) &\leq &2.6\sqrt{(q+1)\ln (q+1)}+2\thicksim 2.6\sqrt{q\ln q},\quad
q\geq 181;  \label{eq4_mu=3_bound} \\
s_{4}(2,q) &\leq &2.8\sqrt{(q+1)\ln (q+1)}+2\thicksim 2.8\sqrt{q\ln q},\quad
q\geq 125.  \label{eq4_mu=4_bound}
\end{eqnarray}
\end{theorem}

\begin{proof}
Let
\begin{equation}
\widehat{w}=d\sqrt{(2q+2)\ln \left( (q+1)^{2}\right) }.  \label{overline_w_d}
\end{equation}
We first establish some inequalities that will be useful in the
proof below. From (\ref{eq2_pi=_start}), (\ref{eq2_pi_estimate}),
(\ref{T0}), (\ref{Rwq}), (\ref{w<(q+1)/2}), and
(\ref{overline_w_d}), it is easy to see that
\begin{align}
& q+1-2w>0,\text{ }2(q+f-w)>q+1\text{ if }f\geq 1,\text{ }\widehat{w}\leq w<
\widehat{w}+1,\text{ }  \label{eq4_useful} \\
& R_{w,q}<e^{-\frac{w^{2}}{2q+2}}\leq e^{-\frac{\text{ }\widehat{w}^{2}}{2q+2
}},\text{ }w^{2}\pm w<2\text{ }\widehat{w}^{2},\text{ }(w-2)(w-1)w(w+1)<2
\widehat{w}^{4},  \notag \\
& (w-1)w(w+1)<3\widehat{w}^{3},\text{ }(w-4)(w-3)(w-2)(w-1)w(w+1)<\text{ }
\widehat{w}^{6}.  \notag
\end{align}
A set $\mathcal{K}_{w+1}$ covers $A$ exactly once if one line
through $A$ contains two points of $\mathcal{K}_{w+1}$, whereas
each of the remaining $q$ lines contains at most one point of
$\mathcal{K}_{w+1}.$ So,
\begin{equation}
T_{1}=(q+1)\binom{q}{2}\cdot q^{w-1}\binom{q}{w-1}.  \label{T1}
\end{equation}

A set $\mathcal{K}_{w+1}$ covers $A$ exactly twice if some two
lines through $A$ contains two points of $\mathcal{K}_{w+1}$,
whereas each of the remaining $q-1$  lines contains at most one
point of $\mathcal{K} _{w+1}. $ So,
\begin{equation}
T_{2}=\binom{q+1}{2}\binom{q}{2}^{2}\cdot q^{w-3}\binom{q-1}{w-3}.
\label{T2}
\end{equation}

Finally, a set $\mathcal{K}_{w+1}$ covers $A$ exactly $3$ times in the
following two cases:

- one line through $A$ contains three points of
$\mathcal{K}_{w+1}$, whereas each of the remaining $q$  lines
contains at most one point of $\mathcal{K} _{w+1};$

- three lines through $A$ contain two points of
$\mathcal{K} _{w+1}$, whereas each of the remaining $q-2$ lines
contains at most one point of $\mathcal{K}_{w+1}.$

Therefore,
\begin{equation}
T_{3}=(q+1)\binom{q}{3}\cdot q^{w-2}\binom{q}{w-2}+\binom{q+1}{3}\binom{q}{2}
^{3}\cdot q^{w-5}\binom{q-2}{w-5}.  \label{T3}
\end{equation}

Let $\mu =2.$ Taking into account (\ref{w =}), (\ref{T0}),
(\ref{pi_1,mu}), (\ref{overline_w_d}) -- (\ref{T1}), we have
\begin{align*}
& \pi _{2}=R_{w,q}\left( 1+\frac{T_{1}}{T_{0}}\right) =R_{w,q}\left( 1+\frac{
w(w+1)(q-1)}{2q(q+1-w)}\right) <
e^{-\frac{w^{2}}{2q+2}}\left( 1+\frac{w(w+1)
}{2(q+1-w)}\right) = \\
&\frac{2q+2+w^{2}-w}{2(q+1-w)}e^{-\frac{w^{2}}{2q+2}}< \frac{2q+2+2\widehat{w
}^{2}}{q+1}e^{-\frac{\widehat{w}^{2}}{2q+2}}=\frac{2+8d^{2}\ln \left(
q+1\right) }{(q+1)^{2d^{2}}}.
\end{align*}
By a computer aided computation,
\begin{equation}
(q+1)^{2}\pi _{2}=\frac{2+8d^{2}\ln \left( q+1\right) }{(q+1)^{2d^{2}-2}}<1
\text{ if }d=1.2,\text{ }q\geq 97.  \label{eq4_(q+1)^2*pi2<1}
\end{equation}
Under condition (\ref{eq4_(q+1)^2*pi2<1}), the probability in
(\ref{K_w+1 satur}) is positive. So, taking into account
(\ref{w =}), the upper bound in (\ref{eq4_mu=2_bound}) is
proved.

Let $\mu =3.$ Taking into account (\ref{w =}), (\ref{T0}),
(\ref{pi_1,mu}), (\ref{overline_w_d}) -- (\ref{T2}), we have
\begin{align*}
& \pi _{3}=R_{w,q}\left( 1+\frac{w(w+1)(q-1)}{2q(q+1-w)}+\frac{
(w-2)(w-1)w(w+1)(q-1)^{2}}{8q^{2}(q+2-w)(q+1-w)}\right) < \\
& e^{-\frac{w^{2}}{2q+2}}\left( 1+\frac{w(w+1)}{2(q+1-w)}+\frac{
(w-2)(w-1)w(w+1)}{8(q+2-w)(q+1-w)}\right) < \\
& e^{-\frac{\widehat{w}^{2}}{2q+2}}\left( 1+\frac{2\widehat{w}^{2}}{q+1}+
\frac{2\widehat{w}^{4}}{2(q+1)^{2}}\right) =\frac{1+8d^{2}\ln \left(
q+1\right) +16d^{4}\ln ^{2}\left( q+1\right) }{\left( q+1\right) ^{2d^{2}}}.
\end{align*}
By a computer aided computation,
\begin{equation}
(q+1)^{2}\pi _{3}=\frac{1+8d^{2}\ln \left( q+1\right) +16d^{4}\ln ^{2}\left(
q+1\right) }{\left( q+1\right) ^{2d^{2}-2}}<1\text{ if }d=1.3,\text{ }q\geq
181.  \label{eq4_(q+1)^2*pi3<1}
\end{equation}
Under condition (\ref{eq4_(q+1)^2*pi3<1}), the probability in
(\ref{K_w+1 satur}) is positive. So, taking into account
(\ref{w =}), the upper bound in (\ref{eq4_mu=3_bound}) is
proved.

Finally, let $\mu =4.$ Taking into account (\ref{w =}),
(\ref{T0}), (\ref{pi_1,mu}), (\ref{overline_w_d}) --
(\ref{T3}), we have
\begin{align*}
 &\pi _{4}=\pi _{3}+R_{w,q}\frac{T_{3}}{T_{0}}=\pi _{3}+R_{w,q}\left( \frac{
(w-1)w(w+1)(q-1)(q-2)}{6q^{2}(q+2-w)(q+1-w)}+\right. \\
& \left. \frac{(w-4)(w-3)(w-2)(w-1)w(w+1)(q-1)^{3}}{
48q^{3}(q+3-w)(q+2-w)(q+1-w)}\right) <\pi _{3}+e^{-\frac{\widehat{w}^{2}}{
2q+2}}\left( \frac{3\widehat{w}^{3}}{(q+1)^{2}}+\frac{\widehat{w}^{6}}{
6(q+1)^{3}}\right).
\end{align*}
It can be checked by computer that
\begin{equation}
(q+1)^{2}\pi _{4}<1\text{ if }d=1.4,\text{ }q\geq 125.
\label{eq4_(q+1)^2*pi4<1}
\end{equation}
Under condition (\ref{eq4_(q+1)^2*pi4<1}), the probability in
(\ref{K_w+1 satur}) is positive. So, taking into account
(\ref{w =}), also (\ref{eq4_mu=4_bound}) is proved.
\end{proof}

\section{Upper bounds on the smallest size of a saturating set in the
projective space $PG(N,q)$ \label{sec_space}}

A point set $S\subset PG(N,q) $ is \emph{saturating} if any
point of $PG(N,q)\setminus S$ is collinear with two points in
$S$. Results on saturating sets in $PG(N,q)$ can be found for
instance in \cite
{BrPlWi,DavCovRad2,DGMP-AMC,DMP-JCTA2003,Giul2013Survey,Janwa1990,Ughi}
and the references therein.

Let $[n,n-r]_{q}R$ be a linear $q$-ary code of length $n,$
codimension $r,$ and covering radius$~R.$ The homogeneous
coordinates of the points of a saturating set with size $n$ in
$PG(r-1,q),$ form a parity check matrix of an $[n,n-r]_{q}2$
code; see \cite
{BrPlWi,CHLS-bookCovCod,DavCovRad2,DGMP-AMC,Giul2013Survey,Janwa1990}.
Let $s(N,q)$ be the smallest size of a saturating set in
$PG(N,q),$ $N\geq 3$. In terms of covering codes, we recall the
equality $s(N,q)=\ell (2,N+1,q)$.
\begin{proposition}
\label{prop_space}For the smallest size $s(N,q)$ of a
saturating set in the projective space $PG(N,q)$ and for the length function $\ell (2,N+1,q),$
the following upper bound holds:
\begin{align}
&s(N,q)=\ell (2,N+1,q)\leq \label{eq5_SatSatSpace}\\
&\left(2\sqrt{(q+1)\ln
(q+1)}+2\right)q^{\frac{N-2}{2}}+2q^{\frac{N-4}{2}}\thicksim 2q^{\frac{N-1}{2}}\sqrt{\ln q},~N=2t-2\geq
6,
\notag
\end{align}
where $t=4,6$ and $t\ge8$, $N\neq 8,12$, $q\geq79$.
\end{proposition}

\begin{proof}
By Theorem \ref{th1}, there is a saturating set with size $n_q=2\sqrt{(q+1)\ln
(q+1)}+2$
in $PG(2,q)$. From the corresponding $[n_{q},n_{q}-3]_{q}2$
code, by using the construction of \cite[Ex.\thinspace
6]{DavCovRad2}, see also \cite[Th.\thinspace 4.4]{DGMP-AMC},
one can obtain an $[n,n-r]_{q}2$ code with $r=2t-1\geq
7,~r\neq 9,13,~n=n_{q}q^{t-2}+2q^{t-3}$, under condition $q+1\geq 2n_q$ that holds for $q\ge79$.
\end{proof}

Surveys of the known $[n,n-r]_{q}2$ codes and saturating sets
in $PG(N,q)$ can be found in \cite{DavCovRad2,DGMP-AMC,Giul2013Survey}.
In many cases bound (\ref{eq5_SatSatSpace}) is better
than the known ones.

A point set $S\subset PG(N,q)$ is $(1,\mu) $-\emph{saturating} if
for every point $Q$ of $PG(N,q)\setminus S$ the
number of secants of $S$ through $Q$ is at least $\mu $,
counted with multiplicity. The multiplicity of a secant $\ell $
is computed as ${\binom{{ \#(\ell \,\cap S)}}{{2}}}$ \cite{BDGMP-AMC2015,BDGMP-arXiv2015}.

Let $[n,n-r]_{q}(R,\mu )$ be a linear $q$-ary $(R,\mu )$-MCF
code of length $ n,$ codimension $r,$ and covering radius$~R.$
The points of a $(1,\mu) $ -saturating set with size $n$ in $PG(r-1,q)$
 form a
parity check matrix of an $[n,n-r]_{q}(2,\mu ) $ code; see
\cite{BDGMP-AMC2015,Giul2013Survey,PDBGM-ExtendAbstr}. Let
$s_{\mu }(N,q)$ be the smallest size of a $(1,\mu) $-saturating set
in $ PG(N,q),$ $N\geq 3$.

\begin{proposition}
\label{prop_space_mu}For the smallest size $s_{\mu }(N,q)$ of a
$(1,\mu) $-saturating set in the projective space $PG(N,q)$, $N\geq
4$  even, and for the $\mu $-length function, the following
upper bound holds:
\begin{equation}
s_{\mu }(N,q)=\ell _{\mu }(2,N+1,q)\leq q^\frac{N-2}{2}n_{q,\mu }+\max (3,\mu )
\frac{q^\frac{N-2}{2}-1}{q-1}\thicksim 2D_{\mu }q^{\frac{N-1}{2}}\sqrt{\ln q},
\end{equation}
where $n_{q,\mu }=2D_{\mu }\sqrt{(q+1)\ln (q+1)}+2,$ $D_{\mu }$ is as in
\emph{(\ref{eq1_Dmu})}, $q^{\frac{N-2}{2}}+1-\mu \geq n_{q,\mu }.$
\end{proposition}

\begin{proof}
By Theorem \ref{th3}, there exists a $(1,\mu) $-saturating set with
size $n_{q,\mu}$ in $ PG(2,q)$. We directly apply
\cite[Cor.\thinspace 6.5] {BDGMP-AMC2015} to the corresponding
$[n_{q,\mu },n_{q,\mu }-3]_{q}(2,\mu )$ MCF code and use the
one-to-one correspondence between $(1,\mu)$-saturating sets and
$(2,\mu )$-MCF codes.
\end{proof}

\end{document}